\documentclass{birkjour}
\usepackage{hyperref}
\usepackage{lmodern}  
\usepackage[T1]{fontenc}  
\usepackage{fix-cm}
\usepackage{bm}
 \usepackage{mathabx}
 
 \usepackage{amsthm}
 \usepackage{accents}
 \usepackage{amsmath}
 \usepackage[dvipsnames]{xcolor}
 \usepackage{xcolor}
 \usepackage{mathrsfs}
 \newcommand{\nc}{\mathbb{N}}

\usepackage{hhline}
\newtheorem{theorem}{Theorem}[section]
\newtheorem{defin}[theorem]{Definition}

\newtheorem{rem}{Remark}[section]
\newtheorem{pro}{Proposition}[section]

\newenvironment{defn*}{\begin{definition}}{\end{definition}}

 \numberwithin{equation}{section}

\newcommand{\doubleoverline}[1]{\overline{\overline{#1}}}
\begin{document}

%
%
%
%
%
%
%
%
%

\title[Tame Factorization Property II]
 {Tame Factorization Property II}


\author{Buket Can Bahadır}
\email{canbuket@gmail.com}
\author{Nazl\i\;Do\u{g}an}
\address{Fatih Sultan Mehmet Vakif University, 
    34445 Istanbul, Turkey}
\email{ndogan@fsm.edu.tr}

\subjclass{46A04, 46A45, 46A61, 46A63}

\keywords{Fréchet Spaces, Tame Operators, Topological Invariants}

\dedicatory{Dedicated to the memory of Tosun Terzio\u{g}lu \\ on the 10th anniversary of his passing.}

\begin{abstract}
We investigate the relationship between the tame factorization property, denoted by $\mathfrak{TF}$ and introduced in the companion paper \cite{CDI}, and the DN-$\Omega$ type linear topological invariants of Fr\'echet spaces. Combining the basic properties of $\mathfrak{TF}$ with known characterizations of tameness and boundedness, we obtain several results identifying the triples of Fr\'echet spaces that possess $\mathfrak{TF}$. We further exhibit examples showing that 
tame factorization property is a strictly weaker condition than tameness — indeed, we construct triples possessing $\mathfrak{TF}$ none of whose individual pairs are tame. We then investigate triples consisting of an arbitrary Fr\'echet space $X$, a nuclear Fr\'echet space $Y$ satisfying the properties $\underline{DN}$ and $\Omega$, and a power series space of finite type $\Lambda_1(\mathcal{E})$ or infinite type $\Lambda_\infty(\mathcal{E})$. We show that requiring such a triple to possess the tame factorization property $\mathfrak{TF}$  characterizes the corresponding linear topological invariants of $X$; in some cases this holds without any restriction on $Y$, while in others it requires the coincidence of the approximate diametral dimension of $Y$ with that of $\Lambda_1(\mathcal{E})$ or $\Lambda_\infty(\mathcal{E})$.

\end{abstract}

\maketitle
\section{Introduction}

The question of when every continuous linear operator between two Fréchet spaces is bounded has a long history in the theory of locally convex spaces. Vogt \cite{vogt2} settled this problem by providing a complete characterization of all pairs $(E,F)$ of Fréchet spaces, denoted by $(E,F)\in \mathfrak{B}$, for which every continuous linear operator from $E$ to $F$ is bounded. In particular, he showed that whenever one of the spaces $E$ or $F$ is a power series space, boundedness of the pair $(E,F)$ is characterized by a suitable condition of DN-$\Omega$ type on the other space. Later, this concept was generalized to factorized operators. Building upon Zahariuta's classification of Cartesian products of Köthe spaces \cite{Z1, Z2}, Djakov et al. \cite{djakov1} introduced the \emph{bounded factorization property} $\mathfrak{BF}$ for a triple $(E,G,F)$, which requires all continuous linear operators from $E$ into $F$ factoring through $G$ to be bounded. A complete structural characterization of $\mathfrak{BF}$ for arbitrary Fréchet spaces was established by Terzioğlu and Zahariuta \cite{terzi2}.

Tameness offers an analogous, but genuinely different, regularity condition. Tameness requires that the characteristic of continuity of every operator between Fr\'echet spaces $E$ and $F$ be majorized by a single function, that is, there exists a nondecreasing function $S:\mathbb{N}\to \mathbb{N}$ such that for every continuous linear operator $T\in L(E,F)$, there exists $k_0 \in \mathbb{N}$ with
$$\hspace{1.85in}\pi_T(k)\leq S(k) \qquad \hspace{1.05in}\forall k\geq k_{0},$$
where $\pi_{T}$ is the characteristic of continuity of the operator $T$. The notion first appeared in Hamilton's 1982 work on linearly tame spaces \cite{hamil}, and was given its general definition for pairs of Fr\'echet spaces by Dubinsky and Vogt \cite{vogt1}. Since then, tameness has been studied, both for pairs of power series spaces and for pairs in which one of the spaces is an arbitrary Fr\'echet space \cite{A, can, nyberg, pisek2}. In particular, Piszczek \cite{pisek2} showed that, when one of the spaces is a power series space, tameness of the pair $(E,F)$ is characterized by the DN-$\Omega$ type linear topological invariants on the other space, in a manner similar to Vogt's \cite{vogt2} characterization of $\mathfrak{B}$. This shows that tameness parallels boundedness in the language of linear topological invariants.

In \cite{CDI}, we introduced the \emph{tame factorization property} $\mathfrak{TF}$ as the natural tame counterpart of $\mathfrak{BF}$: a triple $(E,G,F)$ has property $\mathfrak{TF}$ if every operator from $E$ to $F$ factoring through $G$ is tame. We gave a complete characterization of $\mathfrak{TF}$ for arbitrary Fr\'echet spaces and specialized it to triples where one or more of the spaces are K\"{o}the spaces. Furthermore, we proved that a triple of Köthe spaces satisfies the tame factorization property if and only if every operator factoring as a product of two quasi-diagonal operators is tame.

In the present paper, we investigate the relationship between tame factorization property and linear topological invariants, in the same spirit as Vogt's \cite{vogt2} and Piszczek's \cite{pisek2} characterizations of boundedness and tameness. In Section 3, we collect the fundamental properties of property $\mathfrak{TF}$—its behavior under quotients, subspaces. By combining these structural properties with the characterizations of boundedness and tameness due to Vogt and Piszczek, we derive various sufficient conditions for a triple to satisfy $\mathfrak{TF}$. We then exhibit triples possessing $\mathfrak{TF}$ none of whose constituent pairs are tame, showing that $\mathfrak{TF}$ is a genuinely new phenomenon rather than tameness. Section 4 contains our main results. We investigate triples consisting of an arbitrary Fr\'echet space $X$, a nuclear Fr\'echet space $Y$ satisfying properties $\underline{DN}$ and $\Omega$, and a power series space of finite type $\Lambda_1(\mathcal{E})$ or infinite type $\Lambda_\infty(\mathcal{E})$. Without any assumption on the approximate diametral dimension of $Y$, we show in Theorem~\ref{TA1} that $(X, Y, \Lambda_1(\mathcal{E})) \in \mathfrak{TF}$ already forces $(X,Y)$ to be tame, and in Theorem~\ref{TA1-1} that $(X,\Lambda_\infty(\mathcal{E}),Y) \in \mathfrak{TF}$ is equivalent to $(X,\Lambda_\infty(\mathcal{E}),Y) \in \mathfrak{BF}$, and to $X$ having the property $(LB^\infty)$. Under the coincidence $\delta(Y) = \delta(\Lambda_\infty(\mathcal{E}))$, Theorem~\ref{TA2} and the Theorem~\ref{TA22} show that tame and bounded factorization over $\Lambda_\infty(\mathcal{E})$ coincide and are equivalent to $X$ possessing $(LB^\infty)$. The most delicate case is the coincidence $\delta(Y) = \delta(\Lambda_1(\mathcal{E}))$, treated in Theorem~\ref{T1}: since $\Lambda_1(\mathcal{E})$ need not embed as a subspace of $Y$ in this case, we adapt the $r$-local imbedding technique developed by Aytuna \cite{A}, combined with Grothendieck's factorization theorem, to show that $(X,\Lambda_1(\mathcal{E}), Y) \in \mathfrak{TF}$ is equivalent to $X$ having property $\overline{\Omega}$, and that $X$ is isomorphic to a power series space of finite type exactly when $\underline{DN}$ holds on $X$ in addition.
\section{Preliminaries}

\label{chp:preliminaries}

In this section, we present the basic definitions, notation, and terminology that will be used throughout the paper. For further background and related results, we refer the reader to \cite{dub} and \cite{vogt4}.

\subsection{Fréchet Spaces}
A \textit{Fr\'echet space} is a complete Hausdorff locally convex space whose topology is defined by a countable family seminorms $(\|\cdot\|_{n})_{n\in \mathbb{N}}$. A \textit{grading} on a Fr\'echet space $E$ is an increasing sequence of seminorms $(\|\cdot\|_{n})_{n\in \mathbb{N}}$ that defines the topology of $E$. Every Fr\'echet space admits a grading and a Fréchet space together with a fixed grading is called a \textit{graded Fréchet space}.
We refer the reader to \cite{hamil} for further details.
In the sequel, we work exclusively with graded Fr\'echet spaces unless stated otherwise.

Let $E$ be a Fréchet space and let $E_n$ be the completion of the normed space $E / \ker \|\cdot\|_{n}$ with respect to the norm induced by $\|\cdot\|_{n}$. The space $E_n$ is called the \textit{local Banach space} associated with the seminorm $\|\cdot\|_{n}$ for each $n \in \mathbb{N}$.
For each $n\in \mathbb{N}$, the inclusion $\ker\|\cdot\|_{n+1}\subseteq \ker\|\cdot\|_{n}$ induces a natural continuous map $$i^{n+1}_{n}: E_{n+1} \to E_{n},$$ called \textit{the linking map}. Thus, \(E\) can be identified with the projective limit of the projective system \((E_n,i^{n+1}_{n})_{n\in\mathbb N}\). A Fréchet space $E$ is called a \emph{nuclear Fréchet space} if it is the projective limit of a sequence of separable Hilbert spaces $(E_n)_{n\in\mathbb{N}}$ whose linking maps $i^{n+1}_{n}:E_{n+1}\to E_n$ are Hilbert--Schmidt operators.

Power series spaces are the fundamental class of Fr\'echet spaces, serving as building blocks for the structure theory of more general Fr\'echet spaces, and they play an essential role in the present paper. For a detailed survey, see \cite{TTV}.

 Let $\alpha=\left(\alpha_{n}\right)_{n\in \mathbb{N}}$ be a non-negative increasing sequence with $\displaystyle \lim_{n\rightarrow \infty} \alpha_{n}=+\infty$. The \textit{power series space of finite type} is defined by
$$\Lambda_{1}\left(\alpha\right)=\left\{x=\left(x_{n}\right)_{n\in \mathbb{N}}: |x|_{r}=\left(\sum^{\infty}_{n=1}\left|x_{n}\right|^{2}e^{2r\alpha_{n}}\right)^{\frac{1}{2}}<+\infty \textnormal{ for all } r<1 \right\}$$
and the \textit{power series space of infinite type} is defined by
$$\displaystyle \Lambda_{\infty}\left(\alpha\right)=\left\{x=\left(x_{n}\right)_{n\in \mathbb{N}}: |x|_{r}=\left(\sum^{\infty}_{n=1}\left|x_{n}\right|^{2}e^{2r\alpha_{n}}\right)^{\frac{1}{2}}<+\infty \textnormal{ for all }r\in \mathbb{R}\right\}.$$
The sequence $\alpha$ is called \textit{exponent sequence}. The nuclearity of a power series spaces $\Lambda_{1}\left(\alpha\right)$ and $\Lambda_{\infty}\left(\alpha\right)$ is characterized by the conditions $\displaystyle \lim_{n\rightarrow \infty} \frac{\ln(n)}{\alpha_{n}}=0$ and $\displaystyle \sup_{n\in \mathbb{N}} \frac{\ln(n)}{\alpha_{n}}<+\infty$, respectively. An exponent sequence $\alpha$ is called \textit{finitely nuclear} if $\Lambda_{1}(\alpha)$ is nuclear.

An exponent sequence $\alpha$ is called
\begin{center}
\begin{itemize}
\item[]{}\textit{$\hspace{1in}$stable} $\hspace{.8in}$if $\hspace{0.25in}\displaystyle \sup_{n\in \mathbb{N}}\hspace{0.025in}\frac{\alpha_{\hspace{0.01in}2n}}{ \alpha_{n}}<+\infty$, ~\\~\\
	\item[]{}\textit{$\hspace{1in}$weakly-stable} $\hspace{0.3in}$ if $ \hspace{0.25in}\displaystyle \sup_{n\in \mathbb{N}}\hspace{0.025in}\frac{\alpha_{\hspace{0.01in}n+1}}{\alpha_{n}}<+\infty$. $\hspace{0.15in}$
\end{itemize}
\end{center}

\subsection{Diametral Dimension and Approximate Diametral Dimension}
Let $E$ be a Fréchet space and $U$ and $V$ be absolutely convex subsets of $E$ such that $U$ absorbs $V$, that is, $V\subseteq CU$ for some $C>0$, and  let $L$ be a subspace of $E$. We define
$$\delta\left(V, U, L\right)=\inf\left\{t>0: V\subseteq tU+L\right\}.$$
The $n^{th}$ \textit{Kolmogorov diameter} of $V$ with respect to $U$ is defined by
$$\hspace{1in} d_{n}\left(V,U\right)=\inf\left\{\delta\left(V, U, L\right): \dim L\leq n\right\}\hspace{0.35in}n=0,1,2,...$$
Let $(U_p)_{p\in\mathbb{N}}$ be a decreasing neighborhood basis at zero in the Fréchet space $E$. The \textit{diametral dimension} of $E$ is defined by
$$\Delta\left(E \right)= \left\{\left(t_{n}\right)_{n\in \mathbb{N}}: \forall p\in \mathbb{N} \hspace{0.1in} \exists \hspace{0.025in} q>p \hspace{0.1in}\lim_{n\rightarrow\infty} t_{n}d_{n}\left(U_{q},U_{p}\right)=0\right\},$$
and the \textit{approximate diametral dimension} of a Fr\'echet space $E$ is defined by
$$\delta\left(E \right)= \left\{\left(t_{n}\right)_{n\in \mathbb{N}}: \exists p\in \mathbb{N} \hspace{0.05in} \forall \hspace{0.025in} q>p \hspace{0.05in}\lim_{n\rightarrow\infty}\frac{ t_{n}}{d_{n}\left(U_{q},U_{p}\right)}=0\right\}.$$
Both the diametral dimension 
$\Delta(E)$ and the approximate diametral dimension $\delta(E)$  are independent of the choice of the basis of zero neighborhoods. For further details, properties, and applications of the diametral dimension and approximate diametral dimension, we refer the reader to \cite{BRP,DFW, ND1, M2, TTK}.

The diametral dimension and the approximate diametral dimension of power series spaces satisfy
\[
\begin{aligned}
\Delta\big(\Lambda_{1}(\alpha)\big) &= \Lambda_{1}(\alpha), 
& \qquad \Delta\big(\Lambda_{\infty}(\alpha)\big) &= \Lambda_{\infty}(\alpha)' , \\
\delta\big(\Lambda_{1}(\alpha)\big) &= \Lambda_{1}(\alpha)', 
& \qquad \delta\big(\Lambda_{\infty}(\alpha)\big) &= \Lambda_{\infty}(\alpha).
\end{aligned}
\]
Here, $\Lambda_1(\alpha)'$ and $\Lambda_\infty(\alpha)'$ denote the strong duals of the corresponding power series spaces.


\subsection{Operators on Fr\'echet Spaces}

A linear operator $T:E\to F$ is said to be \textit{continuous} if for every $k\in \nc$, there exist an $m_{k}\in \nc$ and a $C_{k}>0$ such that
$$\|T x\|_{k}\leq C_{k}\|x\|_{m_{k}}$$
for every $x\in E$.  We denote by $L(E,F)$ the space of all continuous linear operators from $E$ to $F$. 

An operator $T\in L(E,F)$ is said to be \textit{bounded} if there exists a neighborhood $U$ of zero in $E$ such that $T(U)$ is a bounded subset in $F$. We write $(E,F) \in \mathfrak{B}$ to mean that every continuous linear operator from $E$ into $F$ is bounded.

Let $r,k\in \mathbb{N}$ and $T\in L(E,F)$. We define
\[ \|T\|_{k,r} = \sup_{\|x\|_r\leq 1}\|Tx\|_k\]
for every $k,r \in \mathbb{N}$. While $\|T\|_{k,r}$ may fail to be finite for arbitrary $k,r \in \mathbb{N}$, since $T \in L(E,F)$, there exists a function $\phi: \mathbb{N} \to \mathbb{N}$ with $\|T\|_{k,\phi(k)} < \infty$ for every $k \in \mathbb{N}$.

The \textit{characteristic of  continuity map} of a linear operator $T$, $\pi_{T}:\nc\to\nc$, is defined as
$$\hspace{0.5in}\pi_T(k)= \inf\lbrace r\in \nc: \sup_{\|x\|_r\leq 1}\|Tx\|_k<\infty \rbrace= \inf_{r \in \nc}\|T\|_{k,r}, \hspace{0.5in}\forall \;k\in \mathbb{N}.$$
We remark that an operator $T$ is bounded precisely when its characteristic of continuity map is bounded. 

\subsection{Tame Spaces}

Generally, no a priori estimate exists for the characteristic of continuity of a continuous linear operator. Tameness of an operator offers a mechanism to control the unpredictable behavior of characteristic of continuity, defined as follows:

\begin{defin}
Let $E$ and $F$ be Fr\'echet spaces, and $T \in L(E,F)$.   $T$ is called $S$-tame for a nondecreasing function $S:\mathbb{N}\to\mathbb{N}$ if there exists a $k_0 \in \mathbb{N}$ such that
$$\hspace{1.8in}\pi_T(k)\leq S(k),\hspace{1.45in} \forall k\geq k_{0}.$$ We call $T$ linearly tame if $S$ is linear.
\end{defin}

Such phenomena occur naturally for operators between power series spaces. Vogt \cite{vogt5} established that every continuous linear operator between finite type power series spaces is linearly tame.

\begin{defin} Let $E$ and $F$ be Fr\'echet spaces. The pair $(E,F)$ is called tame, denoted by $(E,F)\in \mathfrak{T}$ if there exists a nondecreasing function $S:\mathbb{N}\to \mathbb{N}$ such that for every continuous linear operator $T\in L(E,F)$,  there exists $k_0 \in \mathbb{N}$ with
$$\hspace{1.85in}\pi_T(k)\leq S(k),\hspace{1.45in} \forall k\geq k_{0}.$$ 
If $E=F$, the space $E$ is called tame.
Moreover, if $\mathcal{U}\subset L(E,F)$ is a set of continuous linear operators and every $T\in \mathcal{U}$ satisfies the above condition, then we say that the set $\mathcal{U}$ is \emph{tame}.
\end{defin}
 The second author \cite{ND3,ND4} proved that the families of Toeplitz operators between power series spaces are $S$-tame for a suitable function $S$, whereas the families of Hankel operators are compact. Consequently, both families are $S$-tame.

\begin{rem}

\begin{itemize}
\item[i.] The definition of tameness is independent of the choice of seminorms on $E$ and $F$.
\item[ii.] The tameness of a pair $(E,F)$ is also characterized by the following equivalent condition:  
there exists a sequence of nondecreasing functions $(S_\alpha)_{\alpha\in\mathbb{N}}$, such that for every $T\in L(E,F)$ there exists an $\alpha\in \nc$ such that
$$\hspace{1.85in}\pi_T(k)\leq S_{\alpha}(k),\hspace{1.45in} \forall k\in \nc.$$ 
\end{itemize}
\end{rem}

\subsection{DN-$\bf\Omega$ Type Invariants}

In this subsection, we recall several linear topological invariants of type $\text{DN}$ and $\Omega$. Vogt and Wagner \cite{V1,V2,V3} introduced and employed the diametral dimension and DN-$\Omega$ type linear topological invariants precisely in order to understand the subspaces and quotient spaces of nuclear stable power series spaces. These invariants also serve as crucial tools for characterizing the tameness and boundedness of Fr\'echet space pairs.

\begin{defin}
A Fr\'echet space $(E, \left\|.\right\|_{k})_{k\in \mathbb{N}}$ is said to have the property:
\begin{itemize}
\item[\color{white}.] $\hspace{-0.35in}$ \textnormal{(\textbf{DN})} if there exists a $p\in \mathbb{N}$, so that for each $k\in \mathbb{N}$ an $n\in \mathbb{N}$ and a $C>0$ exist such that
$$\hspace{1.55in} \|x\|^{2}_{k}\leq C\|x\|_{p}\|x\|_{n}\hspace{1.55in}\forall x\in E,$$
\item[\color{white}.] $\hspace{-0.35in}$ {$\bf{(\underline{DN})}$} if there exists a $p\in \mathbb{N}$, so that for each $k\in \mathbb{N}$ an $n\in \mathbb{N}$, $0<\lambda<1$ and $C>0$ exist such that
$$\hspace{1.55in} \|x\|_{k}\leq C\|x\|^{1-\lambda}_{p}\|x\|^{\lambda}_{n}\hspace{1.45in}\forall x\in E,$$
\item[\color{white}.] $\hspace{-0.3in}$ $\bf{(LB_{\infty})}$ if for every positive nondecreasing sequence $\rho_{n}$ which tends to infty, there exists a $p\in \mathbb{N}$, so that for each $k\in \mathbb{N}$ an $n_{0}\in \mathbb{N}$ and $C>0$ exist such that
$$\hspace{1.55in} \|x\|^{1+\rho_{n}}_{k}\leq C\|x\|^{\rho_{n}}_{p}\|x\|_{n}\hspace{1.3in}\forall x\in E$$
for every $k\leq n\leq n_{0}$.
\end{itemize}
\end{defin}
The following relationships hold between these invariants (see \cite{vogt4}, \cite{vogt2}):
$$LB_{\infty}\hspace{0.1in}\Rightarrow \hspace{0.1in} DN \hspace{0.1in} \Rightarrow \hspace{0.1in} \underline{DN}.$$

Power series spaces of finite type satisfy the property $\underline{DN}$, while power series spaces of infinite type satisfy the property $DN$ (see \cite[Lemma 29.2 and Lemma 29.12]{vogt4})
\begin{defin} A Fr\'echet space $(E, \left\|.\right\|_{k})_{k\in \mathbb{N}}$ is said to have the property:
\begin{itemize}
\item[\color{white}.] $\hspace{-0.35in}$ $\bf{(\Omega)}$ if for every $p\in \mathbb{N}$, there exists a $q\in \mathbb{N}$ so that for every $k\in \mathbb{N}$ there exist ~\\ {\color{white}.}$ \hspace{-.1in}$ a $0<\theta<1$ and a $C>0$ satisfying
$$\hspace{1.45in}\|y\|^{*}_{q}\leq C(\|y\|^{*}_{p})^{1-\theta} (\|y\|^{*}_{k})^{\theta} \hspace{1.2in}\forall y\in E^{\prime},$$
\item[\color{white}.] $\hspace{-0.35in}$ $\bf{(\overline{\Omega})}$ if for every $p\in \mathbb{N}$, there exists a $q\in \mathbb{N}$ so that for every $k\in \mathbb{N}$ there exist ~\\ {\color{white}.} $\hspace{-.15in}$ a $C>0$ satisfying
$$\hspace{1.45in}(\|y\|^{*}_{q})^{2}\leq C\|y\|^{*}_{p} \|y\|^{*}_{k} \hspace{1.45in}\forall y\in E^{\prime},$$

\item[\color{white}.] $\hspace{-0.35in}$ $\bf{(\doubleoverline{\Omega})}$ if for every $p\in \mathbb{N}$, there exists a $q\in \mathbb{N}$ so that for every $k\in \mathbb{N}$ and $\varepsilon>0$ ~\\ {\color{white}.} $\hspace{-.15in}$ a $C>0$ satisfying
$$\hspace{1.45in}(\|y\|^{*}_{q})^{1+\varepsilon}\leq C(\|y\|^{*}_{p})^{\varepsilon} \|y\|^{*}_{k} \hspace{1.15in}\forall y\in E^{\prime},$$

\item[\color{white}.] $\hspace{-0.35in}$ $\bf{(LB^{\infty})}$  for every positive nondecreasing sequence $\rho_{n}$ which tends to  infinity  and for every $p\in \mathbb{N}$, there exists a $q\in \mathbb{N}$ so that for every $k\in \mathbb{N}$ there  exist an $n_{0}\in \mathbb{N}$ and a $C>0$ satisfying 
$$\hspace{1.45in}(\|y\|^{*}_{q})^{1+\rho_{n}}\leq C(\|y\|^{*}_{p})^{\rho_{n}} \|y\|^{*}_{k} \hspace{1.05in}\forall y\in E^{\prime},$$
for every $k\leq n\leq n_{0}$.
\end{itemize}
The following implications between these invariants are known (see \cite{vogt2})
$$\doubleoverline{\Omega}\hspace{0.1in}\Rightarrow \hspace{0.1in}\overline{\Omega} \hspace{0.1in}\Rightarrow \hspace{0.1in} LB^{\infty}\hspace{0.1in}\Rightarrow \hspace{0.1in} \Omega.$$
\end{defin}

We note that power series spaces of infinite type satisfy the property $\Omega$, whereas those of finite type satisfy $\overline{\Omega}$ (see \cite[Lemma 29.11 and Lemma 29.12]{vogt4}).


Vogt \cite{vogt2} proved that, if either $E$ or $F$ is a power series space, then $(E,F)\in\mathfrak{B}$ if and only if the other space satisfies an appropriate DN or $\Omega$ type condition. Similarly, Piszczek \cite{pisek2} characterized the tameness of the pair
$(E,F)$ in the case where one of the spaces is a power series space, showing that
$(E,F)$ is tame if and only if the other space satisfies the corresponding topological invariants. Therefore, these invariants play a fundamental role in our study. For the reader's convenience, we present a concise summary of these results below; for a more detailed discussion, we refer to \cite{pisek2,vogt2}.

\begin{theorem}\label{VP} Let $\alpha$ be a weakly-stable sequence and let $E$ be a Fréchet space.
\begin{itemize}
\item[\color{white}.]{\color{white}.}$\hspace{-0.4in}$
\textnormal{1.} \textnormal{(\cite[5.2 Satz]{vogt2} and \cite[Theorem 4.2]{pisek2})}
$$(E,\Lambda_{\infty}(\alpha))\in \mathfrak{T} \hspace{0.1in} \Leftrightarrow \hspace{0.1in} (E,\Lambda_{\infty}(\alpha))\in \mathfrak{B} \hspace{0.1in} \Leftrightarrow \hspace{0.1in} \text{ E has } (LB^{\infty}) .$$
\item[\color{white}.]{\color{white}.}$\hspace{-0.4in}$ \textnormal{2.} \textnormal{(\cite[3.2 Satz]{vogt2} and \cite[Theorem 4.10]{pisek2})}
$$(\Lambda_{\infty}(\alpha),E)\in \mathfrak{T} \hspace{0.1in} \Leftrightarrow \hspace{0.1in} (\Lambda_{\infty}(\alpha),E)\in \mathfrak{B} \hspace{0.1in} \Leftrightarrow \hspace{0.1in}  \text{E has } (LB_{\infty}).$$
\item[\color{white}.]{\color{white}.}$\hspace{-0.4in}$ \textnormal{3.} \textnormal{(\cite[2.1 Satz]{vogt2})}  
$\qquad \qquad\;(\Lambda_{1}(\alpha),E)\in \mathfrak{B}  \hspace{0.1in} \Leftrightarrow \hspace{0.1in}$ E has DN.
~\\{\color{white}.}$\hspace{-0.25in}$ \textnormal{(\cite[Theorem 4.9]{pisek2})}  $\qquad\;(\Lambda_{1}(\alpha),E)\in \mathfrak{T}  \hspace{0.15in} \Leftrightarrow \hspace{0.1in}  \text{E has }  \underline{DN}$.
\item[\color{white}.]{\color{white}.}$\hspace{-0.4in}$ \textnormal{4.} \textnormal{(\cite[4.2 Satz]{vogt2})} 
$\qquad\qquad\;(E,\Lambda_{1}(\alpha))\in \mathfrak{B}  \hspace{0.1in} \Leftrightarrow \hspace{0.1in}$ E has $\doubleoverline{\Omega}$.
~\\{\color{white}.}$\hspace{-0.25in}$ {\textnormal{(\cite[Theorem 4.1]{pisek2})}} Assume that either $E$ is nuclear or $E$ is isomorphic to ~\\{\color{white}.}$\hspace{-0.25in}$ $\lambda(A)$  a Köthe space or $\Lambda_{1}(\alpha)$ is nuclear. Then, $$(E,\Lambda_{1}(\alpha))\in \mathfrak{T}\hspace{0.15in} \Leftrightarrow \hspace{0.1in}  \text{E has } \overline{\Omega}.$$

\end{itemize}
\end{theorem}

\subsection{Nuclear Fr\'echet Spaces with the properties $\underline{\textnormal{DN}}$ and $\Omega$}

The linear topological invariants $\underline{\textnormal{DN}}$ and $\Omega$ play a fundamental role in the structural theory of Fr\'echet spaces. The class of nuclear Fr\'echet spaces possessing both properties is exceptionally rich and encompasses two major families of function spaces arising in complex analysis and partial differential equations:

\begin{enumerate}
    \item \textbf{Spaces of Holomorphic Functions:} Let $\mathcal{O}(M)$ denote the space of analytic functions on a $d$-dimensional Stein manifold $M$, equipped with the compact-open topology. The space $\mathcal{O}(M)$ is a nuclear Fr\'echet space. In particular, $\mathcal{O}(\mathbb{C}^{d})$ is isomorphic to $\Lambda_{\infty}(n^{1/d})$ and $\mathcal{O}(\mathbb{D}^{d})$ is isomorphic to $\Lambda_{1}(n^{1/d})$, where $\mathbb{D}^{d}$ denotes the $d$-dimensional unit polydisc. It is known that $\mathcal{O}(M)$ is isomorphic to a subspace of $\mathcal{O}(\mathbb{D}^{k})$ for some $k$, which implies that $\mathcal{O}(M)$ enjoys the property $\underline{\textnormal{DN}}$. Furthermore, by the Oka-Cartan theorem, $\mathcal{O}(M)$ is isomorphic to a quotient space of $\mathcal{O}(\mathbb{C}^{N})$ for some $N$, and hence it possesses the property $\Omega$. For a comprehensive treatment of the topological properties of $\mathcal{O}(M)$ and related function spaces, we refer to \cite{A1, A2, A, AKT2, R1, TTA} and the references therein.

    \item \textbf{Solution Spaces of Elliptic Partial Differential Operators:} Let $P(D)$ be an elliptic linear partial differential operator with constant coefficients on $\mathbb{R}^{d}$ ($d\geq 2$), and let $O$ be an open subset of $\mathbb{R}^d$. The space of solutions 
    $
    \mathcal{N}(O) = \lbrace f\in C^{\infty}(O) : P(D)f = 0\rbrace$ 
    is a nuclear Fr\'echet space. Vogt \cite{vogtA} showed that $\mathcal{N}(O)$ always possesses property $\Omega$, and further proved that $\mathcal{N}(O)$ satisfies property $\underline{\textnormal{DN}}$ whenever $O$ is connected. Detailed structural properties and related characterizations of such solution spaces can be found in \cite{DK, K1, L1, L2, Wie, TTA, vogtA}.
\end{enumerate}

Having introduced these key models, we now recall the fundamental result regarding the diametral dimension of nuclear Fr\'echet spaces with the properties $\underline{DN}$ and $\Omega$. Aytuna et al. \cite[Proposition 1.1]{AKT2} showed that the diametral dimension of such spaces is determined, up to inclusion,
by the diametral dimensions of the associated power series spaces.

\begin{pro}\cite[Proposition 1.1]{AKT2} 
Let $Y$ be a nuclear Fr\'echet space with properties $\underline{\textnormal{DN}}$ and $\Omega$. There exists an exponent sequence $\mathcal{E}=\left(\mathcal{E}_{n}\right)$ (unique up to equivalence) satisfying:
\begin{equation}
\Delta\left(\Lambda_{1}\left(\mathcal{E}\right)\right)\subseteq\Delta\left(Y\right)\subseteq\Delta\left(\Lambda_{\infty}\left(\mathcal{E}\right)\right). 
\end{equation}
Furthermore, $\Lambda_{1}\left(\alpha\right)\subseteq\Delta\left(Y\right)$ implies $\Lambda_{1}\left(\alpha\right)\subseteq \Lambda_{1}\left(\mathcal{E}\right)$, and $\Delta\left(Y\right)\subseteq \Lambda'_{\infty}\left(\alpha\right)$ implies $\Lambda'_{\infty}\left(\mathcal{E}\right)\subseteq\Lambda'_{\infty}\left(\alpha\right)$.
\end{pro}

The sequence $\mathcal{E}$ in the above proposition is called the \textbf{associated exponent sequence} of $Y$. We note that $\Lambda_{\infty} (\mathcal{E})$ is always nuclear, whereas $\Lambda_{1} (\mathcal{E})$ need not be nuclear (see \cite{AKT2} for details). In particular, for the concrete models introduced above, the associated exponent sequence is given by $\mathcal{E} = (n^{1/d})$ for $\mathcal{O}(M)$ and $\mathcal{E} = (n^{1/(d-1)})$ for $\mathcal{N}(O)$, (see \cite[p. 318, 320]{TTA} for detailed computations).

By arguments analogous to those developed by Aytuna et al. \cite{AKT2}, one also obtains the corresponding inclusions for the approximate diametral dimension $\delta$:
\begin{equation}
\delta\left(\Lambda_{\infty}\left(\mathcal{E}\right)\right)\subseteq\delta\left(Y\right)\subseteq\delta\left(\Lambda_{1}\left(\mathcal{E}\right)\right),
\end{equation}
see \cite{AKT2,ND2} for further details.

The coincidence of the diametral dimension and/or the approximate diametral dimension of a Fr\'echet space $E$ with that of a power series space is of particular importance as it naturally appears in several concrete settings.
Indeed, in  \cite[Theorem 4.4]{A}, Aytuna established that if $M$ is a Stein
manifold of dimension $d$, then the space $\mathcal{O}(M)$ of holomorphic
functions on $M$ has diametral dimension equal to either
$\Delta(\Lambda_{1}(n^{\frac{1}{d}}))$ or $\Delta(\Lambda_{\infty}(n^{\frac{1}{d}}))$. Such coincidences lead to strong structural consequences for nuclear Fr\'echet spaces.
In particular, Aytuna et al. proved in~\cite[Theorem 1.2]{AKT2} that a nuclear
Fr\'echet space $Y$ with the properties $\underline{DN}$ and $\Omega$ contains a
complemented copy of $\Lambda_{\infty}(\mathcal{E})$, provided that $\Delta(Y)=\Delta(\Lambda_{\infty}(\mathcal{E}))$ and the exponent sequence $\mathcal{E}$ is stable.
Moreover, in ~\cite[Theorem 3.2]{A}, Aytuna showed that  for a nuclear Fr\'echet space $Y$
with the properties $\underline{DN}$ and $\Omega$ and a stable, finitely
nuclear associated exponent sequence $\mathcal{E}$, $Y$ is isomorphic to a power
series space of finite type if and only if $Y$ is tame and
$\delta(Y)=\delta(\Lambda_{1}(\mathcal{E}))$. We note that the second author proved in \cite[Theorem 3.1]{ND1} that
$ \Delta(Y)=\Delta(\Lambda_{\infty}(\mathcal{E}))$ if and only if $\delta(Y)=\delta(\Lambda_{\infty}(\mathcal{E}))$.
However, this equivalence no longer holds in general for $\Lambda_{1}(\mathcal{E})$.
in \cite[Proposition 4.1]{ND1} shows that $\delta(Y)=\delta(\Lambda_{1}(\mathcal{E}))$ implies 
$\Delta(Y)=\Delta(\Lambda_{1}(\mathcal{E}))$. The converse implication fails. Indeed, in~\cite{ND2}, a nuclear Köthe space $Y$ is constructed  such that
$\Delta(Y)=\Delta(\Lambda_{1}(\mathcal{E}))$ but
$\delta(Y)\neq\delta(\Lambda_{1}(\mathcal{E}))$.
For this reason, in Section 4, we formulate our results under assumptions on
the approximate diametral dimension rather than on the diametral dimension.

\section{Tame Factorization Property}\label{sec3}

In this section, we introduce the \textit{tame factorization property} for triples of Fr\'echet spaces and discuss its fundamental properties.

Let $E$, $F$, and $G$ be Fréchet spaces. An operator $T \in L(E,F)$ factors over $G$ if it can be written as $T = RQ$ for some $Q \in L(E,G)$ and $R \in L(G,F)$. The set of all operators that factor over $G$ will be denoted by $L^G(E,F)$.

\begin{defin} 
 Let $E$, $F$ and $G$ be Fr\'echet spaces. We say that the triple $(E,G,F)$ has tame factorization property, denoted by $(E,G,F)\in \mathfrak{TF}$, if there exists a nondecreasing function $S:\mathbb{N}\to \mathbb{N}$ such that for every $T
\in L^G(E,F)$, there exists $k_0 \in \mathbb{N}$ such that
$$\hspace{1.9in}\pi_T(k)\leq S(k)\hspace{1.4in}\forall k\geq k_{0}.$$  
\end{defin}

\begin{rem}\label{Rem1}
Let $E$, $F$ and $G$ be Fr\'echet spaces. The following important statements are straightforward.
\begin{enumerate}
    \item If $(E,G)\in \mathfrak{T}$ and $(G,F)\in \mathfrak{T}$, then $(E,G,F)\in \mathfrak{TF}$.
    \item If $(E, G_{1})\in \mathfrak{T}$ and $(G_{2}, F)\in \mathfrak{T}$, then $(E, G_{1}\times G_{2},F)\in \mathfrak{TF}$. 
    \item $(E,G,F)\in \mathfrak{TF}$ for any G, if $(E,F)\in \mathfrak{T}$.
    \item Let $X$ be a quotient space of $E$, $Y$ a complemented subspace of $G$, and $Z$ a subspace of $F$.
Then, if $(E,G,F) \in \mathfrak{TF}$, it follows that $(X,Y,Z) \in \mathfrak{TF}$.
\end{enumerate}
\end{rem} 

By combining the above remark with known results from the literature, we obtain the following immediate consequences, whose proofs are omitted.

Let $E$ and $F$ be arbitrary Fréchet spaces.

\begin{enumerate}
\item $(\Lambda_{1}(\alpha), E, \Lambda_{1}(\beta)) \in \mathfrak{TF}$, since $(\Lambda_{1}(\alpha), \Lambda_{1}(\beta)) \in \mathfrak{T}$ by Theorem \ref{VP}(3).

\item
The following assertions hold (and remain valid with $\mathfrak{TF}$ replaced by $\mathfrak{BF}$):
\begin{itemize}
\item[(i)] $(E,\Lambda_{1}(\alpha),\Lambda_{\infty}(\beta))\in\mathfrak{TF}$,
\item[(ii)] $(\Lambda_{1}(\alpha),E,\Lambda_{\infty}(\beta))\in\mathfrak{TF}$,
\item[(iii)] $(\Lambda_{1}(\alpha),\Lambda_{\infty}(\beta),E)\in\mathfrak{TF}$.
\end{itemize}
Indeed, every continuous linear operator from
$\Lambda_{1}(\alpha)$ to $\Lambda_{\infty}(\beta)$ is compact by \cite{Z1}.
\end{enumerate}

Assume that the sequence $\alpha$ is weakly-stable.Then the following assertions hold:

\begin{enumerate}
\setcounter{enumi}{2}
\item If $E$ is a nuclear Fréchet space or a Köthe space satisfying $\overline{\Omega}$, and $F$ satisfies $\underline{DN}$, then
\[(E, \Lambda_{1}(\alpha), F) \in \mathfrak{TF},
\]
since $(E, \Lambda_{1}(\alpha)) \in \mathfrak{T}$ and $(\Lambda_{1}(\alpha), F) \in \mathfrak{T}$ .

\item  If $E$ is a nuclear Fréchet space or a Köthe space satisfying $\overline{\Omega}$, 
\[(E, F, \Lambda_{1}(\alpha)) \in \mathfrak{TF}\]
since $(E, \Lambda_{1}(\alpha)) \in \mathfrak{T}$.

\item If $F$ has property $\underline{DN}$, then
\[(\Lambda_{1}(\alpha), E, F) \in \mathfrak{TF},
\]
since $(\Lambda_{1}(\alpha), F) \in \mathfrak{T}$.

\item If $E$ has property $(LB^{\infty})$ and $F$ has property $(LB_{\infty})$, then
\[(E, \Lambda_{\infty}(\alpha), F) \in \mathfrak{TF},
\]
since $(E, \Lambda_{\infty}(\alpha)) \in \mathfrak{T}$ and $(\Lambda_{\infty}(\alpha), F) \in \mathfrak{T}$.

\item If $E$ has property $(LB^{\infty})$, then
\[(E, F, \Lambda_{\infty}(\alpha)) \in \mathfrak{TF},
\]
since $(E, \Lambda_{\infty}(\alpha)) \in \mathfrak{T}$.

\item If $F$ has property $(LB_{\infty})$, then
\[
(\Lambda_{\infty}(\alpha), E, F) \in \mathfrak{TF},
\]
since $(\Lambda_{\infty}(\alpha), F) \in \mathfrak{T}$.
\end{enumerate}
We refer Theorem \ref{VP} for the precise correspondence between each condition and the relation 
$\mathfrak{T}$.

Following the statements in Remark \ref{Rem1}, one is naturally led to consider the following questions.

\begin{itemize}
\item[i.] 
\textit{ Does there exist a triple $(E,G,F)\in \mathfrak{TF}$ such that $(E,G)\notin \mathfrak{T}$ and $(G,F)\notin \mathfrak{T}$?}

An example can be constructed by an important result of Zahariuta. Zahariuta in \cite{Z} showed that every continuous linear operator $T:\Lambda_{1}(\alpha)\to \Lambda_{\infty}(\beta)$ is compact, as a direct consequence, the pair $(\Lambda_{1}(\alpha), \Lambda_{\infty}(\beta))$ is tame. Hence, for any Fr\'echet space $G$, the triple $(\Lambda_{1}(\alpha), G, \Lambda_{\infty}(\beta))$ has tame factorization property. Let $G = G_{1} \oplus G_{2}$, where $G_{1}$ does not satisfy $\underline{DN}$ and $G_{2}$ does not satisfy $\overline{\Omega}$. It follows that $G$ satisfies neither $\underline{DN}$ nor $\overline{\Omega}$.
In this case, we indeed have
$$(\Lambda_{1}(\alpha), G)\notin \mathfrak{T} \qquad \text{and} \qquad (G, \Lambda_{\infty}(\beta))\notin\mathfrak{T}.$$
To give concrete examples, one can consider the Dragilev spaces $L_{f}(a,r)$ where $f$ is chosen to be rapidly increasing. In this case, $L_{f}(a,0)$ does not have property $\underline{DN}$, while $L_{f}(a,1)$ does not have property $\overline{\Omega}$, see \cite{KN} for details. 
\item[ii.] \emph{Is there a triple $(X,Y,Z)\in \mathfrak{TF}$ such that $(X,Z)\notin \mathfrak{T}$?}

Let $\alpha$ be a weakly-stable sequence and let $E$ be a Fréchet space that does not satisfy the $\underline{DN}$ property. Then
\[
(\Lambda_{1}(\alpha),E)\notin \mathfrak{T}.
\]
On the other hand, for any sequence $\beta$, each operator from $\Lambda_{1}(\alpha)$ to $\Lambda_{\infty}(\beta)$
is compact. It follows that
\[
(\Lambda_{1}(\alpha),\Lambda_{\infty}(\beta),E)\in \mathfrak{BF}
\quad \text{and} \quad
(\Lambda_{1}(\alpha),\Lambda_{\infty}(\beta),E)\in \mathfrak{TF}.
\]
\end{itemize}

\section{Characterizing Topological Invariants via the Tame Factorization Property}

When one of the spaces is a power series space, the bounded and tame pairs of Fr\'echet spaces were fundamentally characterized via linear topological invariants by Vogt \cite{vogt2} and Piszczek \cite{pisek2}, respectively. Extending the results of Section~\ref{sec3}, we investigate whether the tame factorization property provides exact characterizations for the topological invariants of an arbitrary Fr\'echet space $X$. To do this, we investigate triples consisting of an arbitrary Fr\'echet space $X$, a nuclear Fr\'echet space $Y$ satisfying properties $\underline{\textnormal{DN}}$ and $\Omega$, and a power series space of finite type $\Lambda_1(\mathcal{E})$ or infinite type $\Lambda_\infty(\mathcal{E})$. We show that requiring such triples to possess the tame factorization property $\mathfrak{TF}$ (or its bounded variant $\mathfrak{BF}$) characterizes the linear topological invariants on the arbitrary Fr\'echet space $X$.

We first treat factorization without any assumption on the approximate diametral dimension of $Y$. We show in Theorem~\ref{TA1} that $(X, Y, \Lambda_1(\mathcal{E})) \in \mathfrak{TF}$ implies $(X,Y)$ to be tame, and in Theorem~\ref{TA1-1} that $(X,\Lambda_\infty(\mathcal{E}),Y) \in \mathfrak{TF}$ is equivalent to $(X,\Lambda_\infty(\mathcal{E}),Y) \in \mathfrak{BF}$, and to $X$ having the property $(LB^\infty)$. Under the coincidence $\delta(Y) = \delta(\Lambda_\infty(\mathcal{E}))$, Theorem~\ref{TA2} and the Theorem~\ref{TA22} show that tame and bounded factorization over $\Lambda_\infty(\mathcal{E})$ coincide and are equivalent to $X$ possessing $LB^\infty$ or $LB_\infty$. The most involved case is the coincidence $\delta(Y) = \delta(\Lambda_1(\mathcal{E}))$, treated in Theorem~\ref{T1}, where we show that $(X,\Lambda_1(\mathcal{E}), Y) \in \mathfrak{TF}$ is equivalent to $X$ having property $\overline{\Omega}$, and that $X$ is isomorphic to a power series space of finite type exactly when $\underline{DN}$ holds on $X$ in addition. Since $\Lambda_1(\mathcal{E})$ need not embed as a subspace of $Y$ in this case, the proof adapts the $r$-local imbedding technique developed by Aytuna \cite{A}, combined with Grothendieck's factorization theorem.

We begin our investigation by considering tame factorizations without any restrictions on the approximate diametral dimension of $Y$.
 
\begin{theorem}\label{TA1} Let $X$ be a Fréchet space and let $Y$ be a nuclear Fr\'echet space whose associated exponent sequence $\mathcal{E}=(\mathcal{E}_{n})_{n\in \mathbb{N}}$ is stable and finitely nuclear and which has the properties $\underline{DN}$ and $\Omega$.  
 If $(X,Y,\Lambda_{1}(\mathcal{E}))\in \mathfrak{TF}$, then $(X,Y)\in \mathfrak{T}.$ Consequently, the tameness of $(X,Y)$ follows in particular when $X$ satisfies the property $\overline{\Omega}$.
\end{theorem}
\begin{proof} Let us assume that $Y$ is a nuclear Fréchet with the properties $\underline{DN}$ and $\Omega$ and the associated exponent sequence $\mathcal{E}$ is stable and finitely nuclear.  $Y$ is isomorphic to a subspace of $\Lambda_{1}(\mathcal{E})$ by \cite[3.2 Satz]{vogt3}. 


Remark~\ref{Rem1},(4) gives that the assumption $(X,Y,\Lambda_{1}(\mathcal{E}))\in \mathfrak{TF}$ implies that $(X,Y,Y)\in \mathfrak{TF}$, and consequently, the pair $(X,Y)$ is tame.

If $X$ has the property $\overline{\Omega}$, then $(X,\Lambda_{1}(\alpha))\in \mathfrak{T}$ by Theorem \ref{VP}(4). It follows that $(X,Y,\Lambda_{1}(\alpha))\in \mathfrak{TF}$, and hence $(X,Y)\in \mathfrak{T}$.
\end{proof}

\begin{theorem}\label{TA1-1} Let $X$ be a Fréchet space and let $Y$ be a nuclear Fr\'echet space whose associated exponent sequence $\mathcal{E}=(\mathcal{E}_{n})_{n\in \mathbb{N}}$ is stable and finitely nuclear and which has the properties $\underline{DN}$ and $\Omega$. The following statements are equivalent:
\begin{enumerate}
\item $(X,\Lambda_{\infty}(\mathcal{E}), Y)\in \mathfrak{BF}$,
\item $(X,\Lambda_{\infty}(\mathcal{E}), Y)\in \mathfrak{TF}$,
\item $(X,\Lambda_{\infty}(\mathcal{E}))\in \mathfrak{T}$  \item $(X,\Lambda_{\infty}(\mathcal{E}))\in \mathfrak{B}$,
\item $X$ has the property $(LB^{\infty})$.
\end{enumerate}
\end{theorem}
\begin{proof} Let us assume that $Y$ is a nuclear Fréchet with the properties $\underline{DN}$ and $\Omega$ and the associated exponent sequence $\mathcal{E}$ is stable and finitely nuclear. It is evident that (1) implies (2).  

Now, suppose that (2) holds. Since $\mathcal{E}$ is stable and finitely nuclear, $\Lambda_{\infty}(\mathcal{E})$ is isomorphic to a subspace of $\Lambda_{1}(\mathcal{E})$, see \cite[3.2 Satz]{vogt3}. Theorem 2.2 of \cite{AKT2} yields that $\Lambda_{\infty}(\mathcal{E})$ is isomorphic to a subspace of $Y$. By Remark~\ref{Rem1},(4), the assumption $(X,\Lambda_{\infty}(\mathcal{E}),Y)\in \mathfrak{TF}$ implies that $(X,\Lambda_{\infty}(\mathcal{E}),\Lambda_{\infty}(\mathcal{E}))\in \mathfrak{TF}$, and consequently, the pair $(X,\Lambda_{\infty}(\mathcal{E}))$ is tame. Therefore, we obtain (3).

The equivalence of (3), (4), and (5) was established by Vogt and Pisczek which can be seen from Theorem \ref{VP}. It is clear that (4) implies (1). This completes the proof.
\end{proof}

From this point on, we impose additional assumptions on the approximate
diametral dimension. We begin with the case
$\delta(Y)=\delta(\Lambda_{\infty}(\mathcal{E}))$.

\begin{theorem}\label{TA2} Let $X$ be a Fr\'echet space and let $Y$ be a nuclear Fr\'echet space which has the properties $\underline{DN}$ and $\Omega$, and whose associated exponent sequence $\mathcal{E}=(\mathcal{E}_{n})_{n\in \mathbb{N}}$ is stable. Assume that $ \delta(Y)=\delta(\Lambda_{\infty}(\mathcal{E}))$. Then the following statements are  equivalent:
\begin{enumerate}
\item $(\Lambda_{\infty}(\mathcal{E}), Y, X)\in \mathfrak{BF}$,
    \item $(\Lambda_{\infty}(\mathcal{E}), Y, X)\in \mathfrak{TF}$,
    \item $(\Lambda_{\infty}(\mathcal{E}), X)\in \mathfrak{T}$,
    \item $(\Lambda_{\infty}(\mathcal{E}), X)\in \mathfrak{B}$
    \item $X$ has the property $(LB_{\infty})$.
\end{enumerate}
Moreover, if any of the above statements holds, then the pair $(Y,X)$ is tame.
\end{theorem}
\begin{proof} It is clear that (1) implies (2) without any assumption on $Y$.

Let us assume that $\mathcal{E}$ is stable and $ \delta(Y)=\delta(\Lambda_{\infty}(\mathcal{E}))$.  Theorem 3.1 of \cite{ND1} ensures that $\Delta(Y)=\Delta(\Lambda_{\infty}(\mathcal{E}))$. Theorem 1.2 of \cite{AKT2} yields that $\Lambda_{\infty}(\varepsilon)$ is isomorphic to a complemented subspace of $Y$. By Remark~\ref{Rem1},(4)  the assumption $(\Lambda_{\infty}(\mathcal{E}),Y,X)\in \mathfrak{TF}$ implies that $(\Lambda_{\infty}(\mathcal{E}),\Lambda_{\infty}(\mathcal{E}),X)\in \mathfrak{TF}$, then, the pair $(\Lambda_{\infty}(\mathcal{E}), X)$ is tame. Thus, we have shown that (2) implies that (3).  The reverse implication, $(3) \Rightarrow (2)$, is straightforward. The equivalence of statements (3), (4) and (5) is established by Theorem 4.10 of \cite{pisek2} and 3.2 Satz of \cite{vogt2}.

Moreover, if any of the above statements holds, then the pair $(Y,X)$ is tame. Indeed, by \cite[Theorem 4]{vogt6}, the space $Y$ is isomorphic to a quotient of $\Lambda_{\infty}(\mathcal{E})$. Therefore, if $(\Lambda_{\infty}(\mathcal{E}), X) \in \mathfrak{T}$, it follows that $(Y,X) \in \mathfrak{T}$ as well.
\end{proof}

The proof of the following theorem follows exactly the same lines as the proofs of Theorem \ref{TA2}, and is therefore omitted.

\begin{theorem}\label{TA22} Let $X$ be a Fréchet space and let $Y$ be a nuclear Fr\'echet space whose associated exponent sequence $\mathcal{E}=(\mathcal{E}_{n})_{n\in \mathbb{N}}$ is stable, and which has the properties $\underline{DN}$ and $\Omega$. Assume that $ \delta(Y)=\delta(\Lambda_{\infty}(\mathcal{E}))$.
Then the following statements are equivalent:
\begin{enumerate}
\item $(X,Y,\Lambda_{\infty}(\mathcal{E}))\in \mathfrak{BF}$,
    \item $(X,Y,\Lambda_{\infty}(\mathcal{E}))\in \mathfrak{TF}$,
    \item $(X,\Lambda_{\infty}(\mathcal{E}))\in \mathfrak{T}$,
    \item $(X,\Lambda_{\infty}(\mathcal{E}))\in \mathfrak{B}$,
    \item $X$ has $(LB^{\infty})$.
\end{enumerate}
\end{theorem}

We now consider the coincidence condition
$\delta(Y)=\delta(\Lambda_{1}(\mathcal{E}))$. 
In \cite[Theorem 3.2]{A}, Aytuna showed that  for a nuclear Fréchet space $Y$
with the properties $\underline{DN}$ and $\Omega$ and a stable, finitely
nuclear associated exponent sequence $\mathcal{E}$, $Y$ is isomorphic to a power
series space of finite type if and only if $Y$ is tame and
$\delta(Y)=\delta(\Lambda_{1}(\mathcal{E}))$. This motivates our following theorem. To establish our following result, we follow the general outline of the proof of Theorem 3.3 in Aytuna \cite{A}. For the convenience of the reader, and in order to keep the argument self-contained, we use the relevant part of Aytuna’s proof and adapt the concluding step to the present framework.

Since it is not known whether $\Lambda_{1}(\mathcal{E})$ embeds as a subspace of $Y$ under the assumption $\delta(Y)=\delta(\Lambda_{1}(\mathcal{E}))$, we rely on the following weaker notion of embedding, introduced by Aytuna, which only requires a one-sided estimate at a single pair of seminorms.

\begin{defin}[\cite{A}]\label{def:rlocal} Let $Y$ be a Fr\'echet space. 
A continuous linear operator $T:\Lambda_{1}(\alpha)\to Y$ is called an \emph{$(r,k)$-local imbedding} if
$$\hspace{1.25in}\exists\; C>0\qquad\hspace{0.15in} \|T(x)\|_{k}\geq C|x|_{r}, \qquad \hspace{0.45in}\forall\, x\in \Lambda_{1}(\alpha).$$
We say that $Y$ \emph{admits an $r$-local imbedding} from $\Lambda_{1}(\alpha)$ if there exists a $k\in \mathbb{N}$ and an $(r,k)$-local imbedding $T:\Lambda_{1}(\alpha)\to Y$.
\end{defin}

We may now state and prove our result concerning the coincidence $\delta(Y)=\delta(\Lambda_{1}(\mathcal{E}))$.

\begin{theorem}\label{T1} Let $X$ be a Fréchet space and let $Y$ be a nuclear Fr\'echet space whose associated exponent sequence $\mathcal{E}=(\mathcal{E}_{n})_{n\in \mathbb{N}}$ is stable and finitely nuclear, and which has the properties $\underline{DN}$ and $\Omega$. Assume that $ \delta(Y)=\delta(\Lambda_{1}(\mathcal{E}))$.
Then the following statements are equivalent:
\begin{enumerate}
    \item $(X,\Lambda_{1}(\mathcal{E}),Y)\in \mathfrak{TF}$,
    \item $(X,\Lambda_{1}(\mathcal{E}))\in \mathfrak{T}$,
    \item $(X,Y)\in \mathfrak{T}$,
    \item $X$ has the property $\overline{\Omega}$.
\end{enumerate}
Moreover, if $X$ has the property $\underline{DN}$ and any of the above statements holds, then  $X$ is isomorphic to a power series space of finite type.
\end{theorem}

\begin{proof} Let Y  be a nuclear Fréchet space with the properties $\underline{DN}$ and $\Omega$ and whose associated exponent sequence $\mathcal{E}$ is stable and finitely nuclear. Assume that $ \delta(Y)=\delta(\Lambda_{1}(\mathcal{E}))$. As a first step, we show that (1) implies (2). Accordingly, we assume that the triple $(X,\Lambda_{1}(\mathcal{E}),Y)$ has the tame factorization property. 
 
 By \cite[3.2 Satz]{vogt3},  $Y$ can be imbedded into $\Lambda_1(\mathcal{E})$ as a closed subspace. We therefore identify  $Y$ with its image in $\Lambda_1(\mathcal{E})$ and work with the grading on $Y$ induced by $(\Lambda_1(\mathcal{E}), \{|\cdot |_r\}_{r})$. For $r<1$, let $\Lambda_{r}[\mathcal{E}]$ be the Hilbert space defined by
\[
\Lambda_{r}[\mathcal{E}]
= \left\{ x=(x_n) : 
|x|_{r} = \left(\sum_{n=1}^{\infty} |x_n|^2 e^{2r\mathcal{E}_n}\right)^{1/2} < \infty \right\}.
\]
 
 Since $\delta(Y)=\delta(\Lambda_{1}(\mathcal{E}))$,  Theorem 2.9 in \cite{A} ensures the existence of an r-local imbedding from $\Lambda_{1}(\mathcal{E})$ into $Y$ for every  $0< r<1$. Let $0<r_{0}<1$ be fixed, and denote by $T:\Lambda_{1}(\mathcal{E}) \to Y$ a corresponding $(r_{0}, r_{k_{0}})$-local imbedding with $r_{k_{0}}<1$. Then there exists a constant $C_{0}>0$ satisfying
$$|Tx|_{r_{k_0}} \geq C_0 |x|_{r_0}$$
for all $x \in \Lambda_1(\mathcal{E})$. 

Initially, we aim to define a local isomorphism $\widehat{T} :\Lambda_{1}(\mathcal{E}_{2n}) \to Y$ as in \eqref{LI1}. To define this operator, we need to establish some preliminary notions. 

Let $e_{n}$ denote the sequence $(0,\dots,0,1,0,\dots) $ with the entry 
1 in the n-th position and zeros elsewhere for every $n\in \mathbb{N}$. We set $$a_n =\frac{e_n}{e^{r_0 \mathcal{E}_n}}$$
for every $n\in \mathbb{N}$. The sequence $\{a_{n}\}_{n\in \mathbb{N}}$ forms the canonical  orthonormal basis of $\Lambda_{r_0}[\mathcal{E}_n]$. We then define 
$$g_n=T(a_{n})$$
for every $n\in \mathbb{N}$. Since T is a local imbedding, the sequence $\{g_{n}\}_{n\in \mathbb{N}}$ is finitely linearly independent. 

We select a sequence $\{f_n\}_{n\in \mathbb{N}} \subseteq Y\subseteq\Lambda_1(\mathcal{E})$ satisfying the following conditions:
\begin{enumerate}
\item [1.] $f_n \in \text{span} \{g_1,\dots,g_{2n}\}$ for every $n\in \mathbb{N}$,
\item [2.] $\langle f_n, f_s \rangle_{r_{k_0}} = 0$ for every $1\leq s\leq n-1$ and $n\in \mathbb{N}$,
\item [3.] $\langle f_n, \epsilon_k \rangle_{r_{k_0}} = 0$ for every $1\leq k\leq n$ and  $n\in \mathbb{N}$,
\item [4.] $\langle f_n, f_n \rangle_{r_{k_0}} = 1$ for every $n\in \mathbb{N}$,
\end{enumerate}
where $\langle$, $\rangle_{r_{k_0}}$ is the inner product in $\Lambda_{r_{k_0}}[\mathcal{E}]$. This sequence exists and can be chosen inductively, since for each $n\in \mathbb{N}$, the span of  $\{g_1, \dots, g_{2n}\}$ has $2n$ dimension  while $f_n$ is subject to only  $2n-1$ linear conditions.

Each $f_n$ belongs to $\text{span} \{g_1,\dots,g_{2n}\}$, so
there exist scalars $c^{n}_{1}$,$\dots$,$c^{n}_{2n}$ such that
$$\displaystyle f_n= \sum_{i=1}^{2n} c_i^n g_i.$$ 
It then follows that
$$ 1 = |f_n|_{r_{k_0}}= \left|T\left(\sum^{2n}_{i=1} c^{n}_{i}a_{i}\right)\right|_{r_{k_{0}}} \geq C_0 \left| \sum_{i=1}^{2n} c_i^n a_i \right|_{r_{0}} = C_0 \left( \sum_{i=1}^{2n} |c_i^n|^2 \right)^{\frac{1}{2}}$$
from which we conclude 
\begin{equation}
\label{EB1} \sum_{i=1}^{2n} |c_i^n|^2 \leq \frac{1}{C^{2}_0}
\end{equation}
for every $n \in \mathbb{N}$. 
We fix an $s\in (0,1)$. For each $n\in \mathbb{N}$,
\begin{equation*}
\begin{split}
|f_n|_s & = \left| \sum_{i=1}^{2n} c_i^n g_i \right|_s \leq \sum_{i=1}^{2n} |c_i^n| |g_i|_s \leq C_1 \sum_{i=1}^{2n} |c_i| |a_n|_{\pi_T(s)} \nonumber \\
 & = C_1 \sum_{i=1}^{2n} |c_i^n| \left| \frac{e_i}{e^{r_0 \mathcal{E}_i}} \right|_{\pi_T(s)} = C_1 \sum_{i=1}^{2n} |c_i^n| e^{(\pi_T(s)-r_0)\mathcal{E}_i}
 \end{split}
\end{equation*}
for some $C_1>0$, where $\pi_T$ denotes the characteristic of continuity of $T$ with respect to the connonical gradings of $\Lambda_1(\mathcal{E}).$ Next, we choose $K(s)$ satisfying $\max \{\pi_T(s), r_0\} < K(s) < 1$, and applying the Cauchy–Schwarz inequality and using the inequality (\ref{EB1}), we obtain 
\begin{align*}
    &\sum_{i=1}^{2n} |c_i^n| \   e^{(\pi_T(s)-r_0)\mathcal{E}_i + K(s)\mathcal{E}_i - K(s)\mathcal{E}_i} \\
    &\leq  e^{(K(s)-r_0)\mathcal{E}_{2n}} \left( \sum_{i=1}^{2n} |c_i^n|^2 \right)^{\frac{1}{2}} \left( \sum_{i=1}^{2n} e^{2(\pi_T(s)-K(s))\mathcal{E}_i} \right)^{\frac{1}{2}} \\
    &\leq  C e^{(K(s)-r_0)\mathcal{E}_{2n}},
\end{align*}
for some constant $C=C(s,T)$. We note that the second factor in the second line is uniformly bounded for all n since the sequence $(\mathcal{E}_{n})_{n\in \mathbb{N}}$
is finitely nuclear. 
Hence, for every $0<s<1$, there exist constants $C=C(s,T)$ and $K(s)<1$ such that 
\begin{align}\label{E2}
    |f_n|_s \leq C e^{(K(s)-r_0)\mathcal{E}_{2n}}
\end{align}
for every $n\in \mathbb{N}$.

For each $n \in \mathbb{N}$, $f_n$ also admits the representation
$$    f_n = \sum_{k=n+1}^\infty \beta_k^n e_k $$
for some scalars $\displaystyle \{\beta_k^n\}^{\infty}_{k=n+1}$. Then, for any $-\infty<r<r_{k_0}$ and $n\in \mathbb{N}$, we have 
\begin{align*}
    |f_n|_r^2 & = \sum_{k=n+1}^\infty |\beta_k^n|^2 e^{2r\mathcal{E}_k} = \sum_{k=n+1}^\infty |\beta_k^n|^2 e^{2r_{k_0}\mathcal{E}_k} e^{2(r-r_{k_0})\mathcal{E}_k}  \\
    & \leq e^{2(r-r_{k_0})\mathcal{E}_{n+1}} |f_n|^2_{r_{k_0}} = e^{2(r-r_{k_0}) \mathcal{E}_{n+1}}.
\end{align*}
Since the sequence $(\mathcal{E}_n)_{n\in \mathbb{N}}$ is stable, there exists a $C_2>0$ such that $$C_2 \mathcal{E}_{2n} \leq \mathcal{E}_{n+1}$$ for all $n\in \mathbb{N}$. Hence, we obtain
\begin{align} \label{E1}
    |f_n|_r^2 \leq e^{2C_2(r-r_{k_0})\mathcal{E}_{2n}}
\end{align}
for every $r<r_{k_{0}}$ and $n\in \mathbb{N}$.

Let us fix an $s_0$ satisfying $-\infty < s_0 < -2/C_2$. It then follows from \eqref{E1} that, for every 
$n\in \mathbb{N}$.
\begin{align}\label{E3}
    |f_n|_{s_0} \leq e^{(C_2 s_0 - C_2 r_{k_0} + r_0)\mathcal{E}_{2n}} e^{-r_0\mathcal{E}_{2n}} \leq e^{-\mathcal{E}_{2n}} e^{-r_0\mathcal{E}_{2n}}
\end{align}
We emphasize that $s_{0}$ depends solely on the associated exponent sequence $\mathcal{E}$. We choose an increasing sequence $\{ K^+(s)\}_s$ with $K(s)<K^+(s)<1,$ for all $s<1$. For any sequence $x=(x_n)_{n\in \mathbb{N}}\in \Lambda_{1}(\mathcal{E}_{2n})$ and $s<1$, (\ref{E2}) yields
\begin{equation*}
\begin{split}
    \sum_{i=1}^\infty |x_i| |f_i|_s e^{r_0 \mathcal{E}_{2i}} &\leq C\sum^{\infty}_{i=1} |x_{i}|e^{(K(s)-r_{0})\mathcal{E}_{2i}} e^{r_{0}\mathcal{E}_{2i}} 
     ~\\ &= C\sum^{\infty}_{i=1} |x_{i}|e^{K^{+}(s)\mathcal{E}_{2i}} e^{(K(s)-K^{+}(s))\mathcal{E}_{2i}} ~\\
    &\leq \tilde{C} \left( \sum^{\infty}_{i=1} |x_i|^2 e^{2K^+(s)\mathcal{E}_{2i}} \right)^{\frac{1}{2}}=\tilde{C}|x|_{K^{+}(s)}
\end{split}
\end{equation*}
for some $\tilde{C}=\tilde{C}(s)$. This gives us that the map
\begin{equation}\label{LI1}
\begin{split}
\widehat{T}: \Lambda_1(&\mathcal{E}_{2n}) \hspace{0.05in} \longrightarrow \hspace{0.05in} Y~\\ 
&e_{n} \hspace{0.2in}\longrightarrow \hspace{0.05in} \widehat{T}(\epsilon_n)= f_n e^{r_0 \mathcal{E}_{2n}}
\end{split}
\end{equation}
defines a continuous linear operator. The operator $\widehat{T}: \Lambda_1(\mathcal{E}_{2n}) \to Y$ also satisfies
\begin{align*}
    |\widehat{T}(x)|_{r_{k_0}} &= \left|\widehat{T}\left( \sum^\infty_{i=1} x_i e_i \right)\right|_{r_{k_0}}  = \left|\sum^\infty_{i=1} x_i f_i e^{r_0 \mathcal{E}_{2i}}\right|_{r_{k_0}} \\
    & = \left( \sum^\infty_{i=1} |x_i|^2 e^{2 r_0 \mathcal{E}_{2i}} \right)^\frac{1}{2} = |x|_{r_0}
\end{align*}
for all $x \in \Lambda_1(\mathcal{E}_{2n})$.
Hence $\widehat{T}$ is an $(r_0, r_{k_0})$-local isomorphism from $\Lambda_1(\mathcal{E}_{2n})$ into $Y$.

Furthermore, $\widehat{T}$ extends continuously from $\Lambda_1(\mathcal{E}_{2n})$ into $\Lambda_{s_0}[(\mathcal{E}_{n})]$. Indeed, for every $x \in \Lambda_{0}((\mathcal{E}_{2n}))$, (\ref{E3}) gives
\begin{equation}\label{E4}
\begin{split}
    |\widehat{T}(x)|_{s_0} &= \left|\widehat{T}\left( \sum^{\infty}_{i=1} x_i \epsilon_i \right)\right|_{s_0}  = \left| \sum^\infty_{i=1} x_i f_i e^{r_0 \mathcal{E}_{2i}} \right|_{s_0} \\
    & \leq \sum^\infty_{i=1} |x_i| |f_i|_{s_0} e^{r_0 \mathcal{E}_{2i}} \leq \sum^\infty_{i=1} |x_i| e^{-\mathcal{E}_{2i}} \\
    & \leq \left( \sum^\infty_{i=1} e^{-2 \mathcal{E}_{2i}} \right)^\frac{1}{2} |x|_0.
\end{split}
\end{equation}

By considering all $0<r<1$,  we obtain a family $\{\widehat{T}_r\}$ of $(r,r_{k_r})$-local isomorphisms from $\Lambda_1\big(\mathcal{E}_{2n}\big)$ into $Y$, each of which preserves the continuity property established in (\ref{E4}).

Using the $\underline{DN}$ property of power series spaces, Vogt showed 5.1. Lemma of \cite{vogt3} that the characteristics of continuity of operators defined between power series spaces are convex. Hence, the characteristic of continuity  $\pi_{\widehat{T}_{r}}$ is convex for all $0<r<1$. It is clear that $\pi_{\widehat{T}_{r}}(t)<1$ for every $t<1$ and (\ref{E4}) gives us that $\pi_{\widehat{T}_{r}}(s_{0})=0$ and for all $0<r<1$. Therefore, for every $\displaystyle s_{0}<s<1$, we define $$\displaystyle \rho(s)=\frac{2(s-s_{0})}{1+s-2s_{0}}$$ so that for every $0< r<1$ we obtain
\begin{equation*}
\begin{split}
\pi_{\widehat{T}_{r}}(s)&\leq (1-\rho(s))\pi_{\widehat{T}_{r}}(s_{0}) +\rho(s)\pi_{\widehat{T}_{r}}\left(\frac{1+s}{2}\right)  \\
& \leq \rho(s).
\end{split}
\end{equation*}
With these preparations in place, we now prove that the pair $(X,\Lambda_{1}(\mathcal{E}_{2n}))$ is tame. Since $(X, \Lambda_{1}(\mathcal{E}), Y)$ possesses the tame factorization property, there exists a sequence $\displaystyle \{S_{k}\}^{\infty}_{k=1}$ of increasing functions from $\mathbb{N}$ into $\mathbb{N}$ such that for all $T\in L(X,Y)$ which factors over $\Lambda_{1}(\mathcal{E})$, there exists a $k\in \mathbb{N}$ such that $\pi (T)\leq S_{k}$.

Let $L$ be a continuous linear operator from $X$ into $\Lambda_{1}(\mathcal{E}_{2n})$. From the previous considerations, the family of continuous linear operators $\{\widehat{T}_{r}\circ L\}_{r<1}$ satisfies
$$\pi_{\widehat{T}_{r}\circ L}(s)\leq \pi_{L}\circ \rho(s)$$
for every $s_{0}<s<1$. Hence, the family $\{\widehat{T}_{r}\circ L\}_{r<1}$  belongs to
\[
\mathcal{F}
:=\left\{
\begin{aligned}
U\in L(X,Y):\;&
U \textnormal{ factors over } \Lambda_{1}(\mathcal{E}_{2n}) \\
&\textnormal{and }
\pi_{U}(s)\leq \pi_{L}\circ \rho(s)
\textnormal{ for all } s>s_{0}
\end{aligned}
\right\}.
\]
We consider on $\mathcal{F}$ the topology induced by the family of seminorms $$ \displaystyle \{\|\cdot\|^{\pi_{L}\circ \rho(s)}_{s}\}_{s>s_{0}}.$$
With respect to this topology,  $\mathcal{F}$ becomes a Fr\'echet space. Let
$$\mathcal {A}:=\{ U\in L(X,Y): \textnormal{U can be factored over}\;\;\Lambda_{1}(\mathcal{E}_{2n}) \}.$$
Clearly, $\mathcal{F}\subseteq \mathcal{A}$.
Since $(X,\Lambda_{1}(\mathcal{E}_{2n}), Y)$ has tame factorization property, then we may write
$$\mathcal{A}=\bigcup_{k=1}^{\infty} \bigcap_{s=1}^{\infty} \{ U\in L(X,Y):  \textnormal{U can be factored over} \hspace{0.05in} \Lambda_{1}(\mathcal{E}_{2n}) \hspace{0.05in} \textnormal{and} \hspace{0.05in} \|U\|^{ S_{k}(s) }_{s}<+\infty\}$$
This representation allows us to endow $\mathcal{A}$ with a natural topology, making it an LF-space by considering the semi-norms $\displaystyle \|\cdot\|^{S_{k}(s)}_{s}$.

Let us assume that a sequence $\displaystyle {\left(U_{n}\right)_{n\in \mathbb{N}}} \subseteq \mathcal{F}$ converges both to $U_{1}\in \mathcal{F}$ and $U_{2}\in \mathcal{A}$. Then, for every $\varepsilon>0$, there exists $k\in \mathbb{N}$ such that for all $s>s_{0}$,
$$\|U_{n}-U_{2}\|^{S_{k}(s)}_{s}\leq \frac{\varepsilon}{2}$$
and 
$$ \|U_{n}-U_{1}\|^{\pi_{L} \circ \rho(s)}_{s}\leq \frac{\varepsilon}{2}. $$
Consequently, for every $s>s_{0}$ and for every $x\in X$, we have
~\\
\begin{equation*}
\begin{split}
\|U_{2}(x)-U_{1}(x)\|_{s}&\leq \|U_{2}(x)-U_{n}(x)\|_{s}+ \|U_{1}(x)-U_{n}(x)\|_{s} \\ \\
&\leq \|U_{2}-U_{n}\|^{S_{k}}_{s}\|x\|_{S_{k}}+\|U_{1}-U_{n}\|^{\pi_{L}\circ \rho(s)}_{s}\|x\|_{\pi_{L}\circ \rho(s)}\\ \\
&\leq \varepsilon \max\left(\|x\|_{S_{k}},\|x\|_{\pi_{L}\circ \rho(s)}\right).
\end{split}
\end{equation*}
Since $\varepsilon >0$ is arbitrary, it follows that $U_{2}=U_{1}$. Hence, the inclusion $\mathcal{F}\subseteq \mathcal{A}$ has a sequentially closed graph. By Grothendieck Factorization Theorem (\cite{K}, p.18), there exists a $k\in \mathbb{N}$ such that 
$$\displaystyle \mathcal{F}\subseteq \displaystyle \bigcap_{s=1}^{\infty}\{ U\in L(X,Y):  \textnormal{U can be factored over} \hspace{0.1in} \Lambda_{1}(\mathcal{E}_{2n}) \hspace{0.05in} \textnormal{and} \hspace{0.05in} \|U\|^{ S_{k}(s) }_{s}<+\infty\}. $$
Hence, for this $k\in\mathbb{N}$, we have
$$\displaystyle \pi_{\widehat{T}_{r}\circ L}\leq S_{k}.$$
Since $\widehat{T}_{r}$ is $(r,r_{k_{r}})$-local isomorphism, there exists a $\overline{C}>0$ such that for every $x\in X$
$$\|Lx\|_{r}=\|\widehat{T}_{r}(Lx)\|_{r_{k_{r}}}\leq \overline{C} \|x\|_{S_{k}(r_{k_{r}})}.$$
If we define $\widehat{S}_{k}(r)=S_{k}(r_{k_{r}})$ for all $0<r<1$, the above argument shows that for every $L\in L(X,\Lambda_{1}(\mathcal{E}_{2n}))$, there exists a $k\in \mathbb{N}$ such that
$$\pi_{L}(r)\leq \widehat{S}_{k}(r)$$
for every $r<1$.
Therefore, the pair $(X, \Lambda_{1}(\mathcal{E}_{2n}))$ is tame. 
Since $\mathcal{E}$ is stable, it follows that the pair $(X, \Lambda_{1}(\mathcal{E}_{n}))$ is also tame. Thus, we have shown that (1) implies (2). 

Next, since $Y$ is isomorphic to a closed subspace of $\Lambda_{1}(\mathcal{E}_{n})$, the tameness of the pair $(X, \Lambda_{1}(\mathcal{E}_{n}))$ implies that $(X, Y)$ is tame. This establishes that (2) implies (3).

The implication (3) implies (1) is straightforward, and the equivalence between (2) and (4) follows from Theorem 4.1 of \cite{pisek2}.    

Moreover, if $X$ has the property $\underline{DN}$ and  any of the equivalent statements holds, then $X$ is isomorphic to a power series space of finite type by Proposition 29.18 of \cite{vogt4}.
\end{proof}

\section*{Acknowledgments}
This work has been supported by Fatih Sultan Mehmet Vakıf University Research Projects
Coordination Unit under grant number 26FSMBC1FB008.

\end{document}